\documentclass[11pt]{amsart}
\usepackage{graphics,amsmath,amsthm,amssymb,xcolor,booktabs}
\usepackage{mathtools,hyperref}
\newtheorem{theorem}{Theorem}[section]

\usepackage[a4paper,margin=1.9cm,top=2.5cm,bottom=2.5cm,centering,vcentering]{geometry}

\newtheorem{corollary}[theorem]{Corollary}

\newtheorem{remark}[theorem]{Remark}
\newtheorem{conjecture}[theorem]{Conjecture}
\newtheorem{lemma}[theorem]{Lemma}

\def\qed{\hfill \rule{5pt}{7pt}}

\keywords{Overpartition, theta function, congruence, Modular Form, Radu's Algorithm.}


\begin{document}

\title[Overpartitions into nonmultiples of two integers]{Some Properties of Overpartitions into nonmultiples of two integers}

\author[A. M. Alanazi]{Abdulaziz M. Alanazi}
\address{Department of Mathematics, Faculty of Sciences, University of Tabuk, P.O.Box 741, Tabuk 71491, Saudi Arabia}
\email{am.alenezi@ut.edu.sa}

\author[A. O. Munagi]{Augustine O. Munagi}
\address{School of Mathematics, University of the Witwatersrand, Johannesburg, South Africa}
\email{augustine.munagi@wits.ac.za}

\author[M. P. Saikia]{Manjil P. Saikia}
\address{Mathematical and Physical Sciences division, School of Arts and Sciences, Ahmedabad University, Ahmedabad 380009, Gujarat, India}
\email{manjil@saikia.in}

\thanks{The work of the last author is partially supported by an Ahmedabad University Start-Up Grant (Reference No. URBSASI24A5).}


\begin{abstract}
We consider properties of overpartitions that are simultaneously $\ell$-regular and $\mu$-regular, where $\ell$ and $\mu$ are positive relatively prime integers. We prove a seven-way combinatorial identity related to these overpartitions. We also prove several congruence properties satisfied by this class of partitions (and a further related class) using both generating functions and modular forms with Radu's Algorithm.
\end{abstract}

\maketitle

\section{Introduction}\label{intro}
A partition of a positive integer $n$ is a non-increasing sequence of positive integers $\lambda=(\lambda_1, \lambda_2, \ldots, \lambda_k)$ that sum to $n$. The $\lambda_i$'s are called parts of the partition. 
An overpartition of $n$ is a partition of $n$ where the first occurrence of each part size may be overlined. We denote the number of overpartitions of $n$ by $\overline{p}(n)$, with $\overline{p}(0)=1$. For example, $\overline{p}(3)=8$, which enumerates the following overpartitions
$$(3),\, (\overline{3}),\, (2,1),\, (\overline{2},1),\ (2,\overline{1}),\, (\overline{2},\overline{1}),\, (1,1,1),\, (\overline{1},1,1).$$
The three overpartitions with no overlined parts are the ordinary partitions of 3.
We will also use the alternative notation of a partition: $\lambda=(c_1^{u_1},c_2^{u_2},\ldots,c_r^{u_r})$, where $c_1>c_2>\cdots >c_r>0$ and $r\leq k$.

Given a positive integer $\ell$ a partition $\lambda$ is called $\ell$-regular if no part of $\lambda$ is a multiple of $\ell$.
Munagi and Sellers \cite{MunagiSellers2013} studied combinatorial and arithmetic properties of overpartitions of $n$ with $\ell$-regular overlined parts, denoted by $A_{\ell}(n)$. Alanazi and Munagi \cite{AlanaziMunagi} investigated the combinatorial identities for  $\ell$-regular overpartitions of $n$, denoted by $\overline{R_{\ell}}(n)$. In addition, Alanazi et. al. \cite{AlanaziAlenaziKeithMunagi} studied combinatorial and arithmetic properties of the overpartitions of $n$ with $\ell$-regular non-overlined parts. We also note that Alanazi et. al. \cite{AlanaziMunagiSellers} and Shen \cite{Shen} discussed various arithmetic properties of $\overline{R_{\ell}}(n)$. We also recall the generating function from \cite{AlanaziAlenaziKeithMunagi}
\begin{equation}\label{eq:gf-rast}
   \sum_{n\geq 0}\overline{R_\ell^\ast}(n)q^n = \prod\limits_{n=1}^{\infty }\frac{(1-q^{n\ell})(1+q^n)}{1-q^n} = \frac{f_2f_\ell}{f_1^2},
\end{equation}
where $\overline{R_\ell^\ast}(n)$ counts overpartitions of $n$ wherein non-overlined parts are $\ell$-regular and there are no restrictions on the overlined parts. Alanazi et. al. \cite{AlanaziAlenaziKeithMunagi} proved several interesting results related to this function, some of which were recently extended by Sellers \cite{SellersBAMS}.
The second equality in \eqref{eq:gf-rast} follows from the notation $f^k_n$ is defined by $f^k_n:=\prod_{i\geq 1}(1-q^{in})^k$ for all integers $n,k$ with $n>0$.

In this paper, we extend the previous studies and consider properties of overpartitions that are simultaneously $\ell$-regular and $\mu$-regular, where $\ell$ and $\mu$ are positive relatively prime integers.
Let $\overline{R_{\ell,\mu}}(n)$ be the number of overpartitions of $n$ where no parts are divisible by $\ell$ or $\mu$, where $\gcd (\ell,\mu)=1$. We will also call this function the number of $(\ell,\mu)$-regular overpartitions of $n$.
It is easy to see that the generating function is given by 
\begin{equation}\label{eq-gf}
\sum_{n\geq 0}\overline{R_{\ell,\mu}}(n)q^n=\prod\limits_{n=1}^{\infty }\frac{ (
1+q^{n})( 1-q^{\ell n})( 1-q^{\mu n})( 1+q^{\ell\mu n}) }{( 1-q^{n})( 1+q^{\ell n})( 1+q^{\mu n})( 1-q^{\ell\mu n}) } =\frac{f_2f_\ell^2f_\mu^2f_{2\mu\ell}}{f_1^2f_{2\ell}f_{2\mu}f_{\mu\ell}^2}.
\end{equation}

The paper is organized as follows. In Section \ref{sec:two} we prove a seven-way partition identity using generating functions and combinatorial techniques thus establishing the equivalence of several sets of partitions. In Section \ref{sec:three} we prove some congruences modulo small powers of $2$ satisfied by the functions $\overline{R_{\ell,\mu}}(n)$ and $\overline{R_\ell^\ast}(n)$. Then in Section \ref{sec:four} we prove further congruences for $\overline{R_{\ell,\mu}}(n)$ and $\overline{R_\ell^\ast}(n)$ using modular forms, and end the paper in Section \ref{sec:conc} with concluding remarks.

\section{A General Partition Theorem}\label{sec:two}

Our first result is the following seven-way identity.
\begin{theorem}\label{thm7way}
Let
\begin{itemize}
    \item $A_{\ell,\mu}(\ell n)$ denote the number of $\mu$-regular partitions of $\ell n$ where parts that are non-multiples of $\ell$ appear $\ell$ times and parts that are multiples of $\ell$ appear less than $\ell$ times,
    \item $B_{\ell,\mu}(2n)$ denote the number of $(\ell,\mu)$-regular partitions of $2n$ in which $\ell$ and $\mu$ are odd and odd parts occur with even multiplicities,
    \item $C_{\ell,\mu}(2n)$ denote the number of $\mu$-regular partitions of $2n$ in which odd parts appear with multiplicities $2,4,\ldots,2(\ell-2)$ or $2(\ell-1)$ and even parts appear less than $\ell$ times, where $\ell < \mu$ and $\ell$ and $\mu$ are odd,
    \item $D_{\ell,\mu}(2n)$ denote the number of $(\ell,2\mu)$-regular partitions of $2n$ in which odd parts occur with even multiplicities and parts $\equiv \mu\pmod{2\mu}$ appear at most once, where $\mu$ is even,
    \item $E_{\ell,\mu}(2n)$ denote the number of $(2\mu)$-regular partitions of $2n$ in which odd parts occur with multiplicities $2,4,\ldots,2(\ell-2)$ or $2(\ell-1)$, even parts appear less than $\ell$ times and  parts $\equiv \mu\pmod{2\mu}$ appear at most once, where $\ell < \mu$ and $\mu$ is even, and
    \item $F_{\ell,\mu}(\ell n)$ denote the number of $(\ell^2,\mu)$-regular partitions of $\ell n$ in which parts that are non-multiples of $\ell$ appear either $0$ or $\ell$ times.
\end{itemize}

Then 
$$A_{\ell,\mu}(\ell n)=B_{\ell,\mu}(2 n)=C_{\ell,\mu}(2 n)=D_{\ell,\mu}(2 n)=E_{\ell,\mu}(2 n)=F_{\ell,\mu}(\ell n)=\overline{R_{\ell,\mu}}(n).$$
\end{theorem}

\begin{proof} The generating function for $A_{\ell,\mu}(\ell n)$ is
\begin{align}
{\displaystyle\sum\limits_{n=0}^{\infty}}A_{\ell,\mu}(\ell n) q^{\ell n} \nonumber
&= {\displaystyle\prod\limits_{n=1}^{\infty}} \frac{( 1+q^{\ell(\ell n-1)})\cdots( 1+q^{\ell(\ell n-(\ell-1))}) ( 1+q^{\ell n}+\cdots+q^{(\ell-1)\ell n}) }{( 1+q^{\mu\ell(\ell n-1)})\dots( 1+q^{\mu\ell(\ell n-(\ell-1))}) ( 1+q^{\mu\ell n}+\dots+q^{(\ell-1)\mu\ell n}) }\\  \nonumber
&=  {\displaystyle\prod\limits_{n=1}^{\infty}} \frac{( 1+q^{\ell(\ell n-1)})\cdots( 1+q^{\ell(\ell n-(\ell-1))}) ( 1+q^{\ell n}+\cdots+q^{(\ell-1)\ell n}) }{( 1+q^{\mu\ell(\ell n-1)})\cdots( 1+q^{\mu\ell(\ell n-(\ell-1))}) ( 1+q^{\mu\ell n}+\dots+q^{(\ell-1)\mu\ell n}) }\nonumber \\& \qquad \times \frac{(1+q^{\ell^{2}n})(1+q^{\mu\ell^{2}n})}{(1+q^{\ell^{2}n})(1+q^{\mu\ell^{2}n})}\\ \nonumber
&=  {\displaystyle\prod\limits_{n=1}^{\infty}} \frac{( 1+q^{\ell n}) ( 1+q^{\ell n}+\dots+q^{(\ell-1)\ell n}) (1+q^{\mu\ell^{2}n})}{( 1+q^{\mu\ell n}) ( 1+q^{\mu\ell n}+\dots+q^{(\ell-1)\mu\ell n})(1+q^{\ell^{2}n}) }\\ \nonumber
&=  {\displaystyle\prod\limits_{n=1}^{\infty}} \frac{( 1+q^{\ell n}) ( 1-q^{\ell^2 n})( 1-q^{\mu\ell n}) (1+q^{\mu\ell^{2}n})}{( 1+q^{\mu\ell n})( 1-q^{\ell n})( 1-q^{\mu\ell^2 n}) (1+q^{\ell^{2}n}) }.
\end{align}
Replacing $q^\ell$ by $q$ yields the generating function for $\overline{R_{\ell,\mu}}(n)$, that is, \eqref{eq-gf}.

\begin{align}
{\displaystyle\sum\limits_{n=0}^{\infty}}B_{\ell,\mu}(2 n) q^{2 n} \nonumber
&= {\displaystyle\prod\limits_{n=1}^{\infty}} \frac{(1-q^{2\ell n})(1-q^{2\mu n})(1-q^{2\ell(2n-1)})(1-q^{2\mu(2n-1)}) }{(1-q^{2n})(1-q^{2(2n-1)})(1-q^{2\ell\mu n})(1-q^{2\ell \mu(2n-1)}) }\\  \nonumber
&=  {\displaystyle\prod\limits_{n=1}^{\infty}} \frac{(1+q^{2n})(1-q^{2\ell n})(1-q^{2\mu n})(1+q^{2\ell \mu n}) }{(1-q^{2n})(1+q^{2\ell n})(1+q^{2\mu n})(1-q^{2\ell\mu n}) }.
\end{align}
Then replacing $q^2$ by $q$ yields \eqref{eq-gf}.

\begin{align}
{\displaystyle\sum\limits_{n=0}^{\infty}}C_{\ell,\mu}(2 n) q^{2 n} \nonumber
&= {\displaystyle\prod\limits_{n=1}^{\infty}} \frac{(1+q^{2n}+\dots+q^{(\ell-1)2n})(1+q^{2(2n-1)}+\dots+q^{2(\ell-1)(2n-1)}) }{(1+q^{2\mu n}+\dots+q^{(\ell-1)2\mu n})(1+q^{2\mu(2n-1)}+\dots+q^{2(\ell-1)\mu(2n-1)}) }\\  \nonumber
&=  {\displaystyle\prod\limits_{n=1}^{\infty}} \frac{(1-q^{2\ell n})(1-q^{2\ell(2n-1)})(1-q^{2\mu n})(1-q^{2\mu(2n-1)}) }{(1-q^{2n})(1-q^{2(2n-1)})(1-q^{2\ell\mu n})(1-q^{2\ell\mu(2n-1)}) }\\&={\displaystyle\sum\limits_{n=0}^{\infty}}B_{\ell,\mu}(2 n) q^{2 n}.
\end{align}

The generating function of $2\mu$-regular partitions of $n$ in which odd parts occur with even multiplicities and each part $\equiv \mu\pmod{2\mu}$ appears at most once is
\[\prod\limits_{n=1}^{\infty }\frac{1-q^{2\mu n}}{(1-q^{2n})(1+q^{2\mu n})(1-q^{2(2n-1)})}.
\]
Thus, the generating function for $D_{\ell,\mu}(2n)$ is
\begin{align}
{\displaystyle\sum\limits_{n=0}^{\infty}}D_{\ell,\mu}(2 n) q^{2 n} \nonumber
&= {\displaystyle\prod\limits_{n=1}^{\infty}} \frac{1-q^{2\mu n}}{(1-q^{2n})(1+q^{2\mu n})(1-q^{2(2n-1)})}\\ \nonumber & \quad \times \frac{(1-q^{2\ell n})(1+q^{2\ell\mu n})(1-q^{2\ell(2n-1)})}{1-q^{2\ell\mu n}} \\  \nonumber
&=  {\displaystyle\sum\limits_{n=0}^{\infty}}B_{\ell,\mu}(2 n) q^{2 n}.
\end{align}

Next, we have  
\begin{align}
{\displaystyle\sum\limits_{n=0}^{\infty}}E_{\ell,\mu}(2 n) q^{2 n} \nonumber
&= {\displaystyle\prod\limits_{n=1}^{\infty}} \frac{(1-q^{2\ell n})(1-q^{2\ell(2n-1)}) }{(1-q^{2n})(1-q^{2(2n-1)}) }\times \frac{(1-q^{2\mu n})}{(1-q^{2\ell\mu n})}\times \frac{1+q^{2\ell \mu n}}{1+q^2\mu n}\\  \nonumber
&=  {\displaystyle\prod\limits_{n=1}^{\infty}} \frac{(1+q^{2n})(1-q^{2\ell n})(1-q^{2\mu n})(1+q^{2\ell \mu n}) }{(1-q^{2n})(1+q^{2\ell n})(1+q^{2\mu n})(1-q^{2\ell\mu n}) }.
\end{align}
Then replacing $q^2$ by $q$ gives \eqref{eq-gf}.

Lastly, 
\begin{align}
{\displaystyle\sum\limits_{n=0}^{\infty}}F_{\ell,\mu}(\ell n) q^{\ell n} \nonumber
&= {\displaystyle\prod\limits_{n=1}^{\infty}} \frac{( 1+q^{\ell(\ell n-1)})\dots( 1+q^{\ell(\ell n-(\ell-1))}) (1-q^{\ell^2 n})(1-q^{\mu\ell n}) }{( 1+q^{\mu\ell(\ell n-1)})\dots( 1+q^{\mu\ell(\ell n-(\ell-1))}) (1-q^{\ell n})(1-q^{\mu \ell^2 n}) }\\  \nonumber
&=  {\displaystyle\prod\limits_{n=1}^{\infty}} \frac{(1+q^{\ell n})(1+q^{\mu \ell^2 n}) (1-q^{\ell^2 n})(1-q^{\mu\ell n}) }{(1+q^{\mu \ell n})(1+q^{\ell^2 n}) (1-q^{\ell n})(1-q^{\mu \ell^2 n}) }.
\end{align}
Then replacing $q^\ell$ by $q$ yields \eqref{eq-gf} as desired. This completes the proof.
\end{proof}

\section{Congruence Properties via Analytic \& Combinatorial Techniques}\label{sec:three}
Before stating and proving our resutls, we state the following lemmas, the proofs of which may be found in \cite{MunagiSellers2013}. We recall the definition of Ramanujan's theta function
$$ \varphi(q): = \sum_{n=-\infty}^\infty q^{n^2} = \prod_{n=1}^{\infty} (1+q^{2n-1})^2(1-q^{2n}).$$

\begin{lemma} \label{jas_lemma2} We have
$$\varphi(-q^2)^2 = \varphi(q)\varphi(-q).$$
\end{lemma}

\begin{lemma}\label{jas_lemma4} We have
$$\frac{1}{\varphi(-q)} = \varphi(q)\varphi(q^2)^2\varphi(q^4)^4 \cdots$$
\end{lemma}
 We can rewrite the generating function of $\overline{R_{\ell,\mu}}(n)$ in terms of Ramanujan's theta function $\varphi(q)$ as follows
\begin{equation}\label{genfn_iterated}
\sum_{n\geq 0} \overline{R_{\ell,\mu}}(n) q^n = \frac{\varphi(-q^\ell)\varphi(-q^\mu)}{\varphi(-q)\varphi(-q^{\ell\mu})}.
\end{equation}
Incorporating the results of Lemmas \ref{jas_lemma2} and \ref{jas_lemma4}, we can re-write
\begin{equation}\label{genfn_iteratedx}
\sum_{n\geq 0} \overline{R_{\ell,\mu}}(n) q^n = \frac{\varphi(q)\varphi(q^2)^2\varphi(q^4)^4\cdots\varphi(q^{\ell\mu})\varphi(q^{2\ell\mu})^2\varphi(q^{4\ell\mu})^4\cdots }{\varphi(q^\ell)\varphi(q^{2\ell})^2\varphi(q^{4\ell})^4\cdots \varphi(q^\mu)\varphi(q^{2\mu})^2\varphi(q^{4\mu})^4\cdots }.
\end{equation}
The following easy corollary now follows.
\begin{corollary}\label{cor1}
For all $n\geq 1$, we have $\overline{R_{\ell,\mu}}(n) \equiv 0 \pmod{2}.$
\end{corollary}
\begin{proof}
Since $\varphi(q) = 1+2\sum_{n\geq 1}q^{n^2},$ we know that $\varphi(q) \equiv 1\pmod{2}.$ So \eqref{genfn_iteratedx} gives us
\[
\sum_{n\geq 0}\overline{R_{\ell,\mu}}(n)q^n\equiv 1\pmod 2.
\]
This gives an analytic proof.
\end{proof}
 
We give a combinatorial proof as well. 

\begin{proof}[Combinatorial Proof of Corollary \ref{cor1}]
We can obtain an overpartition by overlining the first occurrence of any distinct part of an ordinary partition $\lambda=(c_1^{u_1}, c_2^{u_2}, \ldots,c_r^{u_r})$ with $c_1>c_2>\cdots>c_r$. Since we may choose to either overline a part or not, the number of overpartitions obtainable from $\lambda$ is $\overline{p}(\lambda) := 2^r$. Thus for all $n\geq 1,\, \overline{p}(n) \equiv 0 \pmod{2}$. 
Analogously, if we consider only partitions $\lambda$ that are simultaneously $\ell$- and $\mu$-regular, we would obtain an even number of $(\ell,\mu)$-regular overpartitions. Thus For all $n\geq 1,$ we have $\overline{R}_{\ell,\mu}(n) \equiv 0 \pmod{2}.$
\end{proof}

We now give a complete modulo $4$ characterization of $\overline{R_{\ell,\mu}}(n)$.
\begin{theorem}\label{cor2}
For all $n\geq 1$,
\begin{itemize}
  \item[(i)] if $\ell$ is a square and $\mu$ is not, then
  $$
\overline{R_{\ell,\mu}}(n) \equiv
\begin{cases}
2 \pmod{4}&\mbox{if } n = k^2 \mbox{ or } n =\mu k^2 , \mbox{\ where $\ell$ and $k$ are relatively prime}; \\
0 \pmod{4} & \mbox{otherwise, }
\end{cases} $$
\item[(ii)] if $\ell$ and $\mu$ are both squares, then
  $$ \overline{R_{\ell,\mu}}(n) \equiv
\begin{cases}
2 \pmod{4}&\mbox{if } n = k^2, \mbox{\ where $k$, $\mu$ and $\ell$ are relatively prime}; \\
0 \pmod{4} & \mbox{otherwise, }
\end{cases} $$
  \item[(iii)] { if neither $\ell$ nor $\mu$ is a square }, then
  $$
\overline{R_{\ell,\mu}}(n) \equiv
\begin{cases}
2 \pmod{4}&\mbox{if } n = k^2 \mbox{ , } n =\ell k^2 \mbox{ , } n =\mu k^2 \mbox{ or } n =\ell\mu k^2; \\
0 \pmod{4} & \mbox{otherwise. }
\end{cases} $$
\end{itemize}
\end{theorem}
\begin{proof}[Analytic Proof of Theorem \ref{cor2}]
Using \eqref{genfn_iteratedx} we have
$$\sum_{n\geq 0} \overline{R_{\ell,\mu}}(n) q^n \equiv \frac{\varphi(q)\varphi(q^{\ell\mu})}{\varphi(q^\ell)\varphi(q^\mu)} \pmod{4},$$
since $\varphi(q^i)^j \equiv 1 \pmod{4}$ for any $j\geq 2.$  Next, we deduce from Lemma \ref{jas_lemma2}, that
$$ \varphi(q) = \frac{\varphi(-q^2)^2}{\varphi(-q)}.$$
Thus
\begin{eqnarray*}
\overline{R_{\ell,\mu}}(n) q^n &\equiv &
\frac{\varphi(q)\varphi(q^{\ell\mu})}{\varphi(q^\ell)\varphi(q^\mu)} \equiv 
\frac{\varphi(q)\varphi(q^\ell\mu)\varphi(-q^\ell)\varphi(-q^\mu)}{\varphi(-q^{2\ell})^2\varphi(-q^{2\mu})^2} \pmod{4} \\
&\equiv &
\varphi(q)\varphi(q^\ell\mu)\varphi(-q^\ell)\varphi(-q^\mu) \pmod{4},
\end{eqnarray*}
since $\varphi(-q^{2\ell})^2 \equiv 1 \pmod{4}.$

Hence,
\begin{eqnarray*}
\sum_{n\geq 0} \overline{R_{\ell,\mu}}(n) q^n
&\equiv &
\varphi(q)\varphi(q^\ell\mu)\varphi(-q^\ell)\varphi(-q^\mu) \pmod{4} \\
& = &
(1+2\sum_{n\geq 1}q^{n^2})(1+2\sum_{n\geq 1}(q^{\ell\mu})^{n^2})(1+2\sum_{n\geq 1}(-q^\ell)^{n^2})(1+2\sum_{n\geq 1}(-q^\mu)^{n^2})\\
&\equiv &
1+2\sum_{n\geq 1}q^{n^2}+2\sum_{n\geq 1}(q^{\ell\mu})^{n^2}+2\sum_{n\geq 1}(-q^\ell)^{n^2}+2\sum_{n\geq 1}(-q^\mu)^{n^2} \pmod{4}\\
&\equiv &
1+2\sum_{n\geq 1}q^{n^2}+2\sum_{n\geq 1}q^{\ell\mu{n^2}}+2\sum_{n\geq 1}q^{\ell{n^2}}+2\sum_{n\geq 1}q^{\mu{n^2}} \pmod{4}.
\end{eqnarray*}
A straightforward interpretation of the last congruence gives the three results as required.
\end{proof}
 
We give a combinatorial proof as well. 

\begin{proof}[Combinatorial Proof of Theorem \ref{cor2}]
Let $m(\ell|n)$ be the number of multiples $\ell$ dividing $n$ and define $\Delta (n,\ell,\mu): = \tau(n)-(m(\ell|n)+m(\mu|n)-m(\ell\mu|n))$, where $\tau(n)$ is the number of divisors of $n$. We claim that 
\begin{equation}\label{claimxy}
\overline{R_{\ell,\mu}}(n) \equiv
\begin{cases}
2 \pmod{4}&\mbox{if } \Delta (n,\ell,\mu)\, \mbox{ is odd}; \\
0 \pmod{4} & \mbox{otherwise}.
\end{cases}
\end{equation}

We decompose $(\ell,\mu)$-regular overpartitions into two classes: those containing a unique part-size and those containing two or more different part-sizes.  By the proof of Corollary \ref{cor1} the latter class has cardinality of the form $m2^r, m>0,r>1$ which is divisble by 4. However, partitions with a single part-size arise from divisors of $n$. Each divisor $d$ of $n$ (excluding $\ell, \mu$ and $\ell\mu$) gives the ordinary partition $(d^{n/d})$ which in turn produces two $(\ell,\mu)$-regular overpartitions. 
Thus $\Delta (n,\ell,\mu)\equiv 1$ (mod 2) if and only if divisors of $n$ contribute an odd number of pairs of $(\ell,\mu)$-regular overpartitions.
Hence \eqref{claimxy} follows.

We will use the easily proved relation $m(\ell|n) = \tau(n/\ell)$.

We consider part (i). In view \eqref{claimxy} it will suffice to find the parity of $\Delta (n,\ell,\mu)$ under each constraint.
If $n=k^2$ and $\ell\nmid n$, then $\tau(n)$ is odd and $m(\ell|n)=0$. Since both $m(\mu|n)$ and $m(\ell\mu|n)$ are even, it follows that $\Delta (n,\ell,\mu)$ is odd. 
Similarly, the case $\mu\nmid n$ implies that $\Delta (n,\ell,\mu)$ is odd. 
If $n=k^2$ and $\mu\mid n$, then $m(\mu|n)$ is even (since $\mu$ is not a square). Hence $\Delta (n,\ell,\mu)$ is odd.
But if $n=\mu k^2$, then $\tau(n)$ is even and $m(\mu|n) = \tau(k^2)$ which is odd.
So $\Delta (n,\ell,\mu)$ is odd.
The proof of the first line of part (i) is complete. 

For the second line we consider the the following negations, given that $\ell$ is a square and $\mu$ is not: (a) $n=k^2$ and $\ell\mid n$, and (b) $n\neq k^2$ and $n\neq \mu k^2$.
In (a) we find that both $\tau(n)$ and $m(\ell|n)$ are odd. Since $m(\mu|n)$ and $m(\ell\mu|n)$ are clearly even, it follows that $\Delta (n,\ell,\mu)$ is even. In (b) it is clear that all the relevant functions are even. This completes the proof of part (i).

The other parts may be proved in a similar manner. For example, the first line of part (iii) may be established by noting that exactly one member of the set  $\{\tau(n),m(\ell|n),m(\mu|n),m(\ell\mu|n)\}$ is odd at a time with all the others being even.
\end{proof}

We close this section by proving a general mod $8$ congruence for the $\overline{R}_\ell^\ast(n)$ function. We will need the following lemma.
\begin{lemma}\label{lemma1}
    We have
    \[
    \sum_{n\geq 0}\overline{R^\ast_6}(3n+2)q^n\equiv 4f_6^3\pmod4.
    \]
\end{lemma}

\begin{proof}
   We need the following identity \cite[Theorem 1]{SellersOver}
    \begin{equation}\label{eq:iden-sellers}
        \frac{f_2}{f_1^2}=\frac{f_6^4f_9^6}{f_3^8f_{18}^3}+2q\frac{f_6^3f_9^3}{f_3^7}+4q^2\frac{f_6^2f_{18}^3}{f_3^6}.
    \end{equation}
    Using \eqref{eq:iden-sellers} in \eqref{eq:gf-rast} we have the following
    \begin{align*}
        \sum_{n\geq 0}\overline{R^\ast_6}(3n+2)q^n&=4\frac{f_2^3f_6^3}{f_1^6}\equiv 4\frac{f_1^6f_6^3}{f_1^6} \equiv 4f_6^3 \pmod 4.
    \end{align*}
\end{proof}

\begin{theorem}
    Let $p\geq 5$ be a prime and let $r$ with $1\leq r\leq p-1$ be such that $\textup{inv}(3,p)\cdot 4\cdot r+1$ is a quadratic nonresidue modulo $p$, where $\textup{inv}(3,p)$ is the inverse of $3$ modulo $p$. Then, for all $n\geq 0$ we have $\overline{R^\ast_6}(3(pn+r)+2)\equiv 0 \pmod 8$.
\end{theorem}

\begin{proof}
We need Jacobi's triple product identity
\begin{equation}\label{eq:jacobi}
    f_1^3=\sum_{j=0}^\infty (-1)^j(2j+1)q^{j(j+1)/2}.
\end{equation}
From Lemma \ref{lemma1} and \eqref{eq:jacobi} we have
\[
\sum_{n\geq 0}\overline{R^\ast_6}(3n+2)q^n\equiv 4\sum_{j=0}^\infty (-1)^j(2j+1)q^{3j(j+1)} \pmod 8.
\]
We are interested in values of the form $\overline{R^\ast_6}(3(pn+r)+2)$, and we want to know whether we can have $pn+r=3j(j+1)$, for some nonnegative integer $j$. If such a representation for $pn+r$ exists, then $r\equiv 3j(j+1) \pmod p$. Since $p\geq 3$, so this is equivalent to $\textup{inv}(3,p)\cdot r\equiv j(j+1) \pmod p$, which in turn is equivalent to $\textup{inv}(3,p)\cdot 4\cdot r +1 \equiv (2j+1)^2 \pmod p$. From our assumption that $\textup{inv}(3,p)\cdot 4\cdot r+1$ is a quadratic nonresidue modulo $p$, the result now follows.
\end{proof}

\section{Congruence Properties via Modular Forms}\label{sec:four}
In this section, we will use the theory of modular forms to find several congruences. The first result is given below.

\begin{theorem}\label{thm:cong}
        For all $n\geq 0$, we have
    \begin{align}
         \overline{R_{2,3}}(9n+6)&\equiv 0 \pmod 6,\label{cong-1}\\
     \overline{R_{4,3}}(6n+3)&\equiv 0 \pmod 6,\\
     \overline{R_{4,3}}(6n+5)&\equiv 0 \pmod{12},\\
          \overline{R_{4,3}}(9n+3)&\equiv 0 \pmod 6,\\
    \overline{R_{4,3}}(12n+7)&\equiv 0 \pmod{24},\\
    \overline{R_{4,3}}(12n+11)&\equiv 0 \pmod{72},\\
    \overline{R_{4,9}}(8n+4)&\equiv 0 \pmod{12},\\
    \overline{R_{4,9}}(12n+4)&\equiv 0 \pmod{12},\label{cong-2}\\
    \overline{R_{4,9}}(12n+8)&\equiv 0 \pmod{72},\label{cong-3}\\
        \overline{R_{4,9}}(16n+8)&\equiv 0 \pmod{24},\\
    \overline{R_{4,9}}(18n+12)&\equiv 0 \pmod{96},\\
\overline{R_{4,9}}(24n+20)&\equiv 0 \pmod{216},\label{cong-6}\\
\overline{R_{4,9}}(18n+15)&\equiv 0 \pmod{48},\\
\overline{R_{4,9}}(96n+80)&\equiv 0 \pmod{864},\label{cong-7}\\
     \overline{R_{8,27}}(36n+15)&\equiv 0 \pmod{24},\\
    \overline{R_{8,27}}(36n+21)&\equiv 0 \pmod{96},\\
\overline{R_{8,27}}(36n+24)&\equiv 0 \pmod{12},\\
     \overline{R_{8,27}}(36n+27)&\equiv 0 \pmod{6},\\
 \overline{R_{16,81}}(36n+33)&\equiv 0 \pmod{48},\\
 \overline{R_{16,81}}(72n+60)&\equiv 0 \pmod{48} .
    \end{align}
\end{theorem}

\begin{theorem}\label{thm:cong2}
     For all $n\geq 0$, we have
    \begin{align}
       \overline{R_{3}^\ast}(9n+4)&\equiv 0 \pmod{12}\label{cong-21},\\
        \overline{R_{3}^\ast}(9n+7)&\equiv 0 \pmod{48}\label{cong-22},\\
    \overline{R_{6}^\ast}(9n+5)&\equiv 0 \pmod{24}\label{cong-23},\\
        \overline{R_{6}^\ast}(9n+8)&\equiv 0 \pmod{96}\label{cong-24}.
       \end{align}
\end{theorem}
\begin{remark}
    Alanazi et. al. \cite{AlanaziAlenaziKeithMunagi} had obtained
    \[
    \overline{R_{3}^\ast}(9n+4) \equiv \overline{R_{3}^\ast}(9n+7) \equiv 0 \pmod 3.
    \]
    While Sellers \cite{SellersBAMS} had very recently obtained
        \[
    \overline{R_{3}^\ast}(9n+4) \equiv \overline{R_{3}^\ast}(9n+7) \equiv 0 \pmod 4.
    \]
    Our proof of the first congruence in Theorem \ref{thm:cong2} is independent of the techniques used in the proofs of the results of Alanazi et. al. and Sellers.
\end{remark}

We can also get congruenes for finer arithmetic progressions, we list only two of them without proof here.
\begin{theorem}\label{thm:cong3}
     For all $n\geq 0$, we have
    \begin{align}
    \overline{R_{6}^\ast}(27n+11)&\equiv 0 \pmod{64}\label{cong-n1},\\
        \overline{R_{6}^\ast}(81n+47)&\equiv 0 \pmod{24}\label{cong-n2}.
       \end{align}
\end{theorem}

Theorems \ref{thm:cong} and \ref{thm:cong2} are proved in Section \ref{sec:proof-thm:cong} using an algorithmic approach, due to Smoot \cite{Smoot}.

The following two congruences will be proved using a manual implementation of an algorithm of Radu \cite{Radu2, Radu}, due to computational difficulties.
\begin{theorem}\label{thm:rd}
    For all $n\geq 0$, we have
    \begin{align}
       \overline{R_{3,5}}(9n+3)&\equiv 0 \pmod 6,\label{cong-8}\\
     \overline{R_{2,5}}(18n+9)&\equiv 0 \pmod 6.\label{cong-9}
    \end{align}
\end{theorem}
\noindent We prove Theorem \ref{thm:rd} in Section \ref{sec:rd}.

\subsection{Preliminaries on the Algorithmic Approach}\label{sub:alg}

In this section we describe our methods for proving Theorems \ref{thm:cong}, \ref{thm:cong2} and \ref{thm:rd}. Theorems \ref{thm:cong} and \ref{thm:cong2} are proved using a Mathematica implementation of an algorithm of Radu \cite{Radu2, Radu} due to Smoot \cite{Smoot}, while Theorem \ref{thm:rd} is proved using Radu's algorithm manually. We first describe Radu's algorithm in Subsection \ref{sec:radualg} and then Smoot's implementation in Subsection \ref{sec:smoot}.

\subsection{Radu's Algorithm}\label{sec:radualg}

For a positive integer $N$, we define the following matrix groups:
\begin{align*}
\Gamma & :=\left\{\begin{bmatrix}
a  &  b \\
c  &  d      
\end{bmatrix}: a, b, c, d \in \mathbb{Z}, ad-bc=1
\right\},\\
\Gamma_{\infty} & :=\left\{
\begin{bmatrix}
1  &  n \\
0  &  1      
\end{bmatrix} \in \Gamma : n\in \mathbb{Z}  \right\}.
\end{align*}
$$\Gamma_{0}(N) :=\left\{
\begin{bmatrix}
a  &  b \\
c  &  d      
\end{bmatrix} \in \Gamma : c\equiv~0\pmod N \right\},$$
and
$$[\Gamma : \Gamma_0(N)]:=N\prod_{\ell\mid N} \left( 1+\dfrac{1}{\ell}\right),$$
where $\ell$ is a prime.

We need some preliminary results, which describe an algorithmic approach to proving partition concurrences, developed by Radu \cite{Radu, Radu2}. For integers $M\ge1$, suppose that $R(M)$ is the set of all the integer sequences \[
(r_\delta):=\left(r_{\delta_1},r_{\delta_2},r_{\delta_3},\ldots,r_{\delta_k}\right)
\]
indexed by all the positive divisors $\delta$ of $M$, where $1=\delta_1<\delta_2<\cdots<\delta_k=M$. For integers $m\ge1$, $(r_\delta)\in R(M)$, and $t\in\{0,1,2,\ldots,m-1\}$, we define the set $P(t)$ as
\begin{align}
\label{Pt} P(t):=&\bigg\{t^\prime \in \{0,1,2,\ldots,m-1\} : t^{\prime}\equiv ts+\dfrac{s-1}{24}\sum_{\delta\mid M}\delta r_\delta \pmod{m} \notag\\
&\quad \text{~for some~} [s]_{24m}\in \mathbb{S}_{24m}\bigg\},
\end{align}
where $[x]_m$ denote the residue class of $x$, $\mathbb{Z}_m^{*}$ denote the set of the invertible elements of $\mathbb{Z}_m$, and $\mathbb{S}_m$ denote the set of the squares of $\mathbb{Z}_m^{*}$.

For integers $N\ge1$, $\gamma:=
\begin{pmatrix}
	a  &  b \\
	c  &  d
\end{pmatrix} \in \Gamma$, $(r_\delta)\in R(M)$, and $(r_\delta^\prime)\in R(N)$, we also define
	\begin{align*}
	p(\gamma)&:=\min_{\lambda\in\{0, 1, \ldots, m-1\}}\dfrac{1}{24}\sum_{\delta\mid M}r_{\delta}\dfrac{\gcd(\delta (a+ k\lambda c), mc)^2}{\delta m},\\
	p^\prime(\gamma)&:=\dfrac{1}{24}\sum_{\delta\mid N}r_{\delta}^\prime\dfrac{\gcd(\delta, c)^2}{\delta}.
	\end{align*}

For integers $m\ge1$; $M\ge1$, $N\ge1$,  $t\in \{0,1,2,\ldots,m-1\}$, $k:=\gcd\left(m^2-1,24\right)$, and $(r_{\delta})\in R(M)$, define $\Delta^{*}$ to be the set of all tuples $(m, M, N, t, (r_{\delta}))$ such that all of the following conditions are satisfied
	
	\begin{enumerate}
		\item[1.] Prime divisors of $m$ are also prime divisors of $N$;
		\item[2.] If $\delta\mid M$, then $\delta\mid mN$ for all $\delta\geq1$ with $r_{\delta} \neq 0$;
		\item[3.] $\displaystyle{24\mid kN\sum_{\delta\mid M}\dfrac{r_{\delta} mN}{\delta}}$;
		\item[4.] $\displaystyle{8\mid kN\sum_{\delta\mid M}r_{\delta}}$;
		\item[5.]  $\dfrac{24m}{\left(-24kt-k{\displaystyle{\sum_{\delta\mid M}}{\delta r_{\delta}}},24m\right)} \mid N$;
  \item[6.] If $2|m$ then either $4|kN$ and $8|\delta N$ or $2|s$ and $8|(1-j)N$, where $\prod_{\delta|M}\delta^{|r_\delta|}=2^s\cdot j$.
	\end{enumerate}

We now state a result of Radu \cite{Radu2}, which we use in completing the proof of Theorem \ref{thm:cong}.
\begin{lemma}\label{Lemma Radu 1}
\textup{\cite[Lemma 4.5]{Radu2}} Suppose that $(m, M, N, t, (r_{\delta}))\in\Delta^{*}$, $(r'_{\delta}):=(r'_{\delta})_{\delta \mid N}\in R(N)$, $\{\gamma_1,\gamma_2, \ldots, \gamma_n\}\subseteq \Gamma$ is a complete set of representatives of the double cosets of $\Gamma_{0}(N) \backslash \Gamma/ \Gamma_\infty$, and  $\displaystyle{t_{\min}:=\min_{t^\prime \in P(t)} t^\prime}$,
\begin{align}
	\label{Nu} \nu:= \dfrac{1}{24}\left( \left( \sum_{\delta\mid M}r_{\delta}+\sum_{\delta\mid N}r_{\delta}^\prime\right)[\Gamma:\Gamma_{0}(N)] -\sum_{\delta\mid N} \delta r_{\delta}^\prime-\frac{1}{m}\sum_{\delta|M}\delta r_{\delta}\right)
	- \frac{ t_{min}}{m},
 	\end{align}
$p(\gamma_j)+p^\prime(\gamma_j) \geq 0$ for all $1 \leq j \leq n$, and $\displaystyle{\sum_{n=0}^{\infty}A(n)q^n:=\prod_{\delta\mid M}f_\delta^{r_\delta}.}$ If for some integers $u\ge1$, all $t^\prime \in P(t)$, and $0\leq n\leq \lfloor\nu\rfloor$, $A(mn+t^\prime)\equiv0\pmod u$  is true,  then for integers $n\geq0$ and all $t^\prime\in P(t)$, we have $A(mn+t^\prime)\equiv0\pmod u$.
\end{lemma}

The following lemma supports Lemma \ref{Lemma Radu 1} in the proof of Theorem \ref{thm:cong}.
\begin{lemma}\label{Lemma Wang 1}\textup{\cite[Lemma 2.6]{Radu2}} Let $N$ or $\frac{N}{2}$ be a square-free integer, then we have
		\begin{align*}
		\bigcup_{\delta\mid N}\Gamma_0(N)\begin{pmatrix}
		1  &  0 \\
		\delta  &  1
		\end{pmatrix}\Gamma_ {\infty}=\Gamma.
		\end{align*}
	\end{lemma}

\subsection{Smoot's Implementation of Radu's Algorithm}\label{sec:smoot}

We will also use Smoot's \cite{Smoot} implementation of Radu's algorithm \cite{Radu2,Radu}, which can be used to prove Ramanujan type congruences. The algorithm takes as an input the generating function
\[
\sum_{n= 0}^\infty a_r(n)q^n=\prod_{\delta|M}\prod_{n= 1}^\infty (1-q^{\delta n})^{r_\delta},
\]
and positive integers $m$ and $N$, with $M$ another positive integer and $(r_\delta)_{\delta|M}$ is a sequence indexed by the positive divisors $\delta$ of $M$. With this input, Radu's algorithm tries to produce a set $P_{m,j}(j)\subseteq \{0,1,\ldots, m-1\}$ which contains $j$ and is uniquely defined by $m, (r_\delta)_{\delta|M}$ and $j$. Then, it decides if there exists a sequence $(s_\delta)_{\delta |N}$ such that
\[
q^\alpha \prod_{\delta|M}\prod_{n=1}^\infty (1-q^{\delta n})^{s_\delta} \cdot \prod_{j^\prime \in P_{m,j}(j)}\sum_{n=0}^\infty a(mn+j^\prime)q^n,
\]
is a modular function with certain restrictions on its behaviour on the boundary of $\mathbb{H}$.

Smoot \cite{Smoot} implemented this algorithm in Mathematica and we use his \texttt{RaduRK} package, which requires the software package \texttt{4ti2}. Documentation on how to intall and use these packages are available from Smoot \cite{Smoot}. We use this implemented \texttt{RaduRK} algorithm to prove Theorem \ref{thm:cong} in the next section.

It is natural to guess that $N=m$ (which corresponds to the congruence subgroup $\Gamma_0(N)$), but this is not always the case, although they are usually closely related to one another. The determination of the correct value of $N$ is an important problem for the usage of \texttt{RaduRK} and it depends on the $\Delta^\ast$ criterion described in the previous subsection. It is easy to check the minimum $N$ which satisfies this criterion by running \texttt{minN[M, r, m, j]}.

\subsection{Proofs of Theorems \ref{thm:cong} and \ref{thm:cong2}}\label{sec:proof-thm:cong}

In this section we prove Theorems \ref{thm:cong} and \ref{thm:cong2} using Smoot's implementation described in Subsection \ref{sec:smoot}.

\begin{proof}[Proof of Theorem \ref{thm:cong}]
 Since the proof of all congruences listed in Theorem \ref{thm:cong} are similar, we only prove \eqref{cong-1} here. The rest of the output can be obtained by visiting \url{https://manjilsaikia.in/publ/mathematica/smoot_lm.nb}. We use the procedure call \[\texttt{RK[12,12,\{-2,3,2,-1,-3,1\},9,6]}\] which gives a straight proof of \eqref{cong-1}. Here we give the output of \texttt{RK}.

\allowdisplaybreaks{
		\begin{align*}
		\texttt{In[1] := } & \texttt{RaduRK[12,12,\{-2,3,2,-1,-3,1\},9,6]}\\
		& \prod_{\texttt{$\delta$|M}} (\texttt{q}^{\delta };\texttt{q}^{\delta })_{\infty }^{\texttt{r}_{\delta }}  = \sum_{\texttt{n=0}}^{\infty } \texttt{a}(\texttt{n})\,\texttt{q}^\texttt{n}\\
		& \fbox{$\texttt{f}_\texttt{1}(\texttt{q})\cdot \prod\limits_{\texttt{j}'\in \texttt{P}_{\texttt{m,r}}(\texttt{j}) } \sum\limits_{\texttt{n=0}}^\infty \texttt{a}(\texttt{mn}+\texttt{j}')\,\texttt{q}^\texttt{n} = \sum\limits_{\texttt{g}\in \texttt{AB}} \texttt{g}\cdot \texttt{p}_\texttt{g}(\texttt{t}) $} \\
		& \texttt{Modular Curve: }\texttt{X}_\texttt{0}(\texttt{N}) \\
		\texttt{Out[2] = }\\
  &\begin{array}{c|c}
 \text{N:} & 12 \\
\hline
 \text{$\{$M,(}r_{\delta })_{\delta |M}\text{$\}$:} & \{12,\{-2,3,2,-1,-3,1\}\} \\
\hline
 \text{m:} & 9 \\
\hline
 P_{m,r}\text{(j):} & \{6\} \\
\hline
 f_1\text{(q):} & \dfrac{(q;q)_{\infty }^5 \left(q^4;q^4\right)_{\infty }^2
   \left(q^6;q^6\right)_{\infty }^9}{q^3 \left(q^2;q^2\right){}_{\infty }
   \left(q^3;q^3\right)_{\infty }^3 \left(q^{12};q^{12}\right)_{\infty }^{12}} \\
\hline
 \text{t:} & \dfrac{\left(q^4;q^4\right)_{\infty }^4 \left(q^6;q^6\right)_{\infty }^2}{q
   \left(q^2;q^2\right)_{\infty }^2 \left(q^{12};q^{12}\right)_{\infty }^4} \\
\hline
 \text{AB:} & \{1\} \\
\hline
 \left\{p_g\text{(t): g$\in $AB$\}$}\right. & \left\{6 t^3+6 t^2-6 t-6\right\} \\
\hline
 \text{Common Factor:} & 6 \\
\end{array}
		\end{align*}}

The interpretation of this output is as follows.

The first entry in the procedure call \texttt{RK[12,12,\{-2,3,2,-1,-3,1\},9,6]} corresponds to specifying $N=12$, which fixes the space of modular functions
\[
M(\Gamma_0(N)):=\text{the algebra of modular functions for $\Gamma_0(N)$}.
\]

 The second and third entry of the procedure call \texttt{RK[12,12,\{-2,3,2,-1,-3,1\},9,6]} gives the assignment $\{M, (r_\delta)_{\delta|M}\}=\{12, (-2,3,2,-1,-3,1)\}$, which corresponds to specifying $(r_\delta)_{\delta|M}=(r_1,r_2,r_4)=(-2,3,2,-1,-3,1)$, so that
 \[
\sum_{n\geq 0}\overline{R}_{2,3}(n)q^n=\prod_{\delta|M}(q^\delta;q^\delta)^{r_\delta}_\infty = \frac{f_2^3f_3^2f_{12}}{f_1^2f_4f_6^3}.
 \]

 The last two entries of the procedure call \texttt{RK[12,12,\{-2,3,2,-1,-3,1\},9,6]} corresponds to the assignment $m=8$ and $j=7$, which means that we want the generating function
 \[
\sum_{n\geq 0}\overline{R}_{2,3}(n)(mn+j)q^n=\sum_{n\geq 0}\overline{R}_{2,3}(n)(9n+6)q^n.
 \]
 So, $P_{m,r}(j)=P_{9,r}(6)$ with $r=(-2,3,2,-1,-3,1)$.

The output $P_{m,r}(j):=P_{9,(-2,3,2,-1,-3,1)}(6)=\{6\}$ means that there exists an infinite product
\[
f_1(q)=\dfrac{(q;q)_{\infty }^5 \left(q^4;q^4\right)_{\infty }^2
   \left(q^6;q^6\right)_{\infty }^9}{q^3 \left(q^2;q^2\right){}_{\infty }
   \left(q^3;q^3\right)_{\infty }^3 \left(q^{12};q^{12}\right)_{\infty }^{12}},
\]
such that
\[
f_1(q)\sum_{n\geq 0}\overline{R}_{2,3}(n)(9n+6)q^n\in M(\Gamma_0(12)).
\]

Finally, the output
\[
t=\dfrac{\left(q^4;q^4\right)_{\infty }^4 \left(q^6;q^6\right)_{\infty }^2}{q
   \left(q^2;q^2\right)_{\infty }^2 \left(q^{12};q^{12}\right)_{\infty }^4}, \quad AB=\{1\}, \quad \text{and}\quad \{p_g\text{(t): g$\in AB$\}},
\]
presents a solution to the question of finding a modular function $t\in M(\Gamma_0(12))$ and polynomials $p_g(t)$ such that
\[
f_1(q)\sum_{n\geq 0}\overline{R}_{2,3}(9n+6)q^n =\sum_{g\in AB}p_g(t)\cdot g
\]
In this specific case, we see that the singleton entry in the set $\{p_g\text{(t): g$\in AB$\}}$ has the common factor $6$, thus proving equation \eqref{cong-1}.
\end{proof}

\begin{remark}
    The interested reader can refer to \cite{Saikia} or \cite{AndrewsPaule2} for some more recent applications of the method.
\end{remark}

\begin{proof}[Proof of Theorem \ref{thm:cong2}]
    Since the proof of Theorem \ref{thm:cong2} is similar to the proof of Theorem \ref{thm:cong}, we just refer the reader to the output file available at \url{https://manjilsaikia.in/publ/mathematica/smoot_rl.nb}
\end{proof}

\subsection{Proof of Theorem \ref{thm:rd}}\label{sec:rd}
We use the material in Subsection \ref{sec:radualg} without commentary. We give the full details for the proof of \eqref{cong-8}, and mention the parameters in Radu's algorithm for \eqref{cong-9}.

 For the purposes of \eqref{cong-8}, it is enough to take \[(m,M,N,t,(r_\delta))=(9, 30, 30, 3, (-2,1,2,2,-1,-1,-2,1)).\] It is routine to check that this choice satisfies the $\Delta^{\ast}$ conditions. By equation \eqref{Pt} we see that $P(t)=\{3\}$. For the choice of $(r_\delta^\prime)=(12,0,0,0,0,0,0,0) $ we see that
\[p\left(\begin{pmatrix}
1 & 0\\
\delta & 1
\end{pmatrix}\right)+p^{\prime}\left(\begin{pmatrix}
1 & 0\\
\delta & 1
\end{pmatrix}\right)\ge 0\quad \text{for all $\delta\mid N$,}
\]
and $\lfloor \nu \rfloor=101$ for $t=3$. So we need to just check the validity of \eqref{cong-8} for all $n\leq \lfloor \nu \rfloor$ and then by Lemmas \ref{Lemma Radu 1} and \ref{Lemma Wang 1} we would have proved our result. This can be checked using Mathematica and hence the result follows. 

For \eqref{cong-9}, we take
\[(m,M,N,t,(r_\delta))=(18, 20, 60, 9, (-2,3,-1,2,-3,1)),\]
which will give us 
\[
P(t)=\{9\}, \quad (r_\delta^\prime)=(54,0,0,0,0,0,0,0,0,0,0,0) \quad \text{and} \quad \lfloor \nu \rfloor  = 1322.
\]
We can then verify the congruence up to $n=1322$ via Mathematica from which the result follows.\qed

\section{Concluding Remarks}\label{sec:conc}

\begin{enumerate}
\item It is possible to give a combinatorial proof of Theorem \ref{thm7way}.
\item A modulo $8$ characterization in the same vein as the modulo $4$ characterization in Section \ref{sec:three} is also possible. We leave it as an open problem.
\item It is desirable to have elementary proofs of the results in Section \ref{sec:four}.
\item A further study of congruences is also desired. To that end, we make the following conjecture.
\begin{conjecture}
    For all $n\geq 0$, $\ell \geq 2$, and $1\leq k \leq \ell$, we have
    \[
    \overline{R_{4,9}}(4\ell n+4k)\equiv 0 \pmod 6.
    \]
\end{conjecture}
\end{enumerate}

\end{document}